\documentclass[12pt]{article}
\usepackage[usenames]{color}
\usepackage{colortbl}
\usepackage{tikz-cd}
\usepackage{mathtools}
\usepackage{amsmath}
\usepackage{amscd,amsmath, amssymb, fancyhdr, mathbbol, dutchcal, url, euscript}
\usepackage[backref=page]{hyperref}
\usepackage{mathtools}

\renewcommand*{\backrefalt}[4]{%
	\ifcase #1 (Not cited.)%
	\or        (Cited on page~#2.)%
	\else      (Cited on pages~#2.)%
	\fi}

\hypersetup{
	colorlinks   = true,
	citecolor    = magenta,
	linkcolor    =blue,
	urlcolor     =magenta	
}

\numberwithin{equation}{section}
\newcommand{\version}{version 1.0,\ \ July, 2022}
\def\eqref#1{(\ref{#1})}
\newcommand{\goth}{\mathfrak}
\newcommand{\g}{{\mathfrak g}}
\newcommand{\ag}{{\mathfrak a}}

\newcommand{\arrow}{{\:\longrightarrow\:}}
\newcommand{\Z}{{\Bbb Z}}
\def\C{{\Bbb C}}
\newcommand{\R}{{\Bbb R}}
\newcommand{\Q}{{\Bbb Q}}
\renewcommand{\H}{{\Bbb H}}

\def\1{\sqrt{-1}\:}

\newcommand{\cntrct}                
{\hspace{2pt}\raisebox{1pt}{\text{$\lrcorner$}}\hspace{2pt}}

\renewcommand{\tilde}{\widetilde}
\renewcommand{\bar}{\overline}
\renewcommand{\phi}{\varphi}
\renewcommand{\epsilon}{\varepsilon}
\renewcommand{\geq}{\geqslant}
\renewcommand{\leq}{\leqslant}

\newcommand{\s}{{\goth s}}

\newcommand{\End}{\operatorname{End}}

\newcommand{\Id}{\operatorname{Id}}

\newcommand{\Lie}{\operatorname{Lie}}

\newcommand{\Tw}{\operatorname{Tw}}

\newcommand{\Mon}{\operatorname{Mon}}
\newcommand{\su}{\goth{su}}
\newcommand{\f}{\goth{f}}

\newcommand{\GL}{\operatorname{GL}}

\newcounter{Mycounter}[section]
\newcounter{lemma}[section]
\newcounter{claim}[section]
\newcounter{sublemma}[section]
\newcounter{corollary}[section]
\renewcommand{\thecorollary}{{Corollary \thesection.\arabic{corollary}}}
\newcommand{\corollary}{%
    \setcounter{corollary}{\value{Mycounter}}
    \refstepcounter{corollary}
    \stepcounter{Mycounter}
    {\noindent \bf \thecorollary:\ }}

\newcounter{theorem}[section]
\renewcommand{\thetheorem}{{Theorem \thesection.\arabic{theorem}}}
\newcommand{\theorem}{%
    \setcounter{theorem}{\value{Mycounter}}
    \refstepcounter{theorem}
    \stepcounter{Mycounter}
    {\noindent \bf \thetheorem:\ }}

\newcounter{conjecture}[section]
\renewcommand{\theconjecture}{{Conjecture \thesection.\arabic{conjecture}}}
\newcommand{\conjecture}{%
    \setcounter{conjecture}{\value{Mycounter}}
    \refstepcounter{conjecture}
    \stepcounter{Mycounter}
    {\noindent \bf \theconjecture:\ }}

\newcounter{proposition}[section]
\renewcommand{\theproposition}
      {{Proposition \thesection.\arabic{proposition}}}
\newcommand{\proposition}{%
    \setcounter{proposition}{\value{Mycounter}}
    \refstepcounter{proposition}
    \stepcounter{Mycounter}
    {\noindent \bf \theproposition:\ }}

\newcounter{definition}[section]
\renewcommand{\thedefinition}
      {{Definition~\thesection.\arabic{definition}}}
\newcommand{\definition}{%
    \setcounter{definition}{\value{Mycounter}}
    \refstepcounter{definition}
    \stepcounter{Mycounter}
    {\noindent \bf \thedefinition:\ }}

\newcounter{example}[section]
\renewcommand{\theexample}{{Example \thesection.\arabic{example}}}
\newcommand{\example}{%
    \setcounter{example}{\value{Mycounter}}
    \refstepcounter{example}
    \stepcounter{Mycounter}
    {\noindent \bf \theexample:\ }}

\newcounter{remark}[section]
\renewcommand{\theremark}{{Remark \thesection.\arabic{remark}}}
\newcommand{\remark}{%
    \setcounter{remark}{\value{Mycounter}}
    \refstepcounter{remark}
    \stepcounter{Mycounter}
    {\noindent \bf \theremark:\ }}

\newcounter{problem}[section]
\newcounter{question}[section]
\newcommand{\proof}{\noindent{\bf Proof:\ }}

\makeatletter

\renewcommand{\leftmark}{{\scriptsize Complex curves in hypercomplex nilmanifolds with $\H$-solvable Lie algebras}}

\@addtoreset{equation}{section} \@addtoreset{footnote}{section}
\makeatother

\def\blacksquare{\hbox{\vrule width 5pt height 5pt depth 0pt}}
\def\endproof{\blacksquare}

\title{Complex curves in hypercomplex nilmanifolds with $\H$-solvable Lie algebras}
\author{Yulia Gorginyan}
\begin{document}
\maketitle
\begin{abstract}
An operator $I$ on a real Lie algebra $\g$ is called a complex structure operator if $I^2=-\Id$ and the $\sqrt{-1}$-eigenspace $\g^{1,0}$ is a Lie subalgebra in the complexification of $\g$. A hypercomplex structure on a Lie algebra $\g$ is a triple of complex structures $I,J$ and $K$ on $\g$ satisfying the quaternionic relations. We call a hypercomplex nilpotent Lie algebra {\bf $\H$-solvable} if there exists a sequence of $\H$-invariant subalgebras 
\begin{equation*}
 \g_1^{\H}\supset\g_2^{\H}\supset\cdots\supset\g_{k-1}^{\H}\supset\g_k^{\H}=0,
\end{equation*}
such that $[\g_i^{\H},\g_i^{\H}]\subset\g^{\H}_{i+1}.$ We give examples of $\H$-solvable hypercomplex structures on a nilpotent Lie algebra and conjecture that all hypercomplex structures on nilpotent Lie algebras are $\H$-solvable. Let $(N,I,J,K)$ be a compact hypercomplex nilmanifold associated to an $\H$-solvable hypercomplex Lie algebra. We prove that, for a general complex structure $L$ induced by quaternions, there are no complex curves in a complex manifold $(N,L)$.

\end{abstract}
\tableofcontents

\section{Introduction}

\subsection{Complex nilmanifolds}

Recall that {\bf a nilmanifold} $N$ is a compact manifold that admits a transitive action of a nilpotent Lie group $G$. Any nilmanifold is diffeomorphic to a quotient of a connected, simply connected nilpotent Lie group $G$ by a cocompact lattice $\Gamma$ \cite{Mal}. 

\medskip
 A complex nilmanifold could be defined in two different ways. {\bf A complex parallelizable nilmanifold} \cite{W} is a compact quotient of a complex nilpotent Lie group by a discrete, cocompact subgroup. This is not the definition we use. We define {\bf a complex nilmanifold} as a quotient of a nilpotent Lie group with a left-invariant complex structure by a left action of a cocompact lattice.

\medskip
Nilmanifolds provide examples of non-K\"ahler complex manifolds. One of the first examples of a non-K\"ahler complex manifold was given by Kodaira \cite{Has1}, see also \cite{Th}. It is a complex surface called {\bf the Kodaira surface} with the first Betti number $b_1=3$. It was proven in \cite{BG} that a complex nilmanifold does not admit a K\"ahler structure unless it is a torus. Moreover, in \cite{Has} Hasegawa proved that a nilmanifold that is not a torus is never homotopically equivalent to any K\"ahler manifold. 

\medskip
\example
{\bf The Kodaira surface} could be obtained as a quotient of the group of matrices of the form 
$$G:=\left\{\begin{pmatrix}
1 & x & z & t \\
0 & 1 & y & 0 \\
0 & 0 & 1 & 0\\
0 & 0 & 0 & 1 
\end{pmatrix}\,:\,x,y,z,t\in\R\right\}$$
by the subgroup $\Gamma$ of the similar matrices with integer entries.

\medskip
\begin{example} The example of a complex nilmanifold that is obtained from a complex Lie group is
{\bf an Iwasawa manifold}. It is a compact quotient of the 3-dimensional complex Heisenberg group $G$ by a cocompact, discrete subgroup $\Gamma$ of the corresponding matrices with the Gaussian integer entries. Unlike the Kodaira surface, the Iwasawa manifold is parallelizable, that is, its tangent bundle is trivial as a holomorphic bundle. 
\end{example}


\subsection{$\H$-solvable Lie algebras}
Let $\g=\Lie G$ denote the Lie algebra of a  nilpotent Lie group $G$. 

\medskip
\begin{definition}\label{complex structure}
A subalgebra $\g^{1,0}\subset\g\otimes_{\R}\C$ which satisfies $\g^{1,0}\oplus\bar{\g^{1,0}}=\g\otimes_{\R}\C$ and $[\g^{1,0},\g^{1,0}]\subset\g^{1,0}$
defines {\bf a complex structure operator} $I\in\End(\g)$. Subspaces $\g^{1,0}$ and $\bar{\g^{1,0}}=\g^{0,1}$ are $\sqrt{-1}$ and $-\sqrt{-1}$ eigenspaces of the operator $I$ respectively.
\end{definition}

\medskip
Let $\H$ be the quaternion algebra. Recall that $\H$ is generated by $I, J$ and $K$ satisfying  the quaternionic relations:
\begin{equation}\label{HH}
I^2=J^2=K^2=-{\Id},\quad IJ=-JI=K.
\end{equation}

\medskip
\begin{definition}
 Let $\g$ be a nilpotent Lie algebra. {\bf A hypercomplex structure} on $\g$ is a triple of endomorphisms $I,J,K\in\End(\g)$ which satisfies the conditions \eqref{HH} and defines the complex structures in the sense of \ref{complex structure}.
\end{definition}

Denote by $\g_i^I:=\g_i+I\g_i$ the smallest $I$-invariant Lie subalgebra which contains the commutator subalgebra $\g_i=[\g_{i-1},\g]$, $i\in\Z_{>0}$. Reformulating the result of S. Salamon \cite[Theorem 1.3]{S}, D. Millionschikov in \cite[Proposition 2.5]{Mil} has shown that 
 \begin{equation*}
 [\g_k^I,\g_k^I]\subset\g_{k+1}^I
 \end{equation*}
 and 
 \begin{equation*}
\g_1^I:=[\g,\g]+I[\g,\g]\ne\g.
 \end{equation*}
Hence, a sequence of complex-invariant Lie subalgebras
\begin{equation*}
    \g\supset\g_1^I\supset\cdots\supset\g_n^I=0
\end{equation*}
 terminates for some $n\in\Z_{>0}$. It is natural to ask a similar question about the hypercomplex nilpotent Lie algebras: is the algebra $\H\g_1:=[\g,\g]+I[\g,\g]+J[\g,\g]+K[\g,\g]$ equal to $\g$ or not?

\medskip
We introduce a notion of {\bf an $\H$-solvable} nilpotent Lie algebra. Define inductively $\H$-invariant Lie subalgebras: $\g_i^{\H}:=\H[\g_{i-1}^{\H},\g_{i-1}^{\H}]$, where $\g^{\H}_1=\H[\g,\g]$.

\medskip
\begin{definition}\label{H-sol}
A hypercomplex nilpotent Lie algebra $\g$ is called {\bf $\H$-solvable} if there exists $k\in\Z_{>0}$ such that
\begin{equation*}
 \g_1^{\H}\supset\g_2^{\H}\supset\cdots\supset\g_{k-1}^{\H}\supset\g_k^{\H}=0.
\end{equation*}
\end{definition}
Such a filtration corresponds to an iterated hypercomplex toric bundle, see \cite{AV}. Clearly, this holds if and only if $\g_{i-1}^{\H}\subsetneq\g^{\H}_{i}$ for any $i\in\Z_{>0}$.

There are no known examples of hypercomplex Lie algebras which are not $\H$-solvable.  

\medskip
\conjecture
 All hypercomplex structures on a nilpotent Lie algebra $\g$ are $\H$-solvable.


\subsection{Examples of $\H$-solvable algebras}

An example of an $\H$-solvable Lie algebra is given by {\bf an abelian complex structure}. Let us recall the definition.

\medskip
\begin{definition} Let $\g$ be a nilpotent Lie algebra with a complex structure.
Suppose that $[\g^{1,0},\g^{1,0}]=0$. This complex structure is called {\bf an abelian complex structure}.
\end{definition}

It was already known that a nilpotent Lie algebra that admits an abelian hypercomplex structure is $\H$-solvable \cite[Proposition 4.5]{AV}, see also \cite[Corollary 3.11]{Rol}. 

The main purpose of this article is to prove the following theorem: 

\medskip
\begin{theorem}
Let $(N,I,J,K)$ be a hypercomplex nilmanifold, and assume that the corresponding Lie algebra $\g$ is $\H$-solvable. Then there are no complex curves in the complex nilmanifold $(N,L)$, where $L=aI+bJ+cK$, $L^2=-\Id$ for all $(a,b,c)\in S^2$ except of a countable set $R\subset S^2$.
\end{theorem}

\medskip
To give an example of an $\H$-solvable Lie algebra with a non-abelian hypercomplex structure, we need a construction described below. {\bf The quaternionic double} was introduced in the work \cite{SV}. Let $(X,I_X)$ be a complex manifold which admits a torsion-free flat connection $\nabla:TX\arrow TX\otimes\Lambda^1X$ which also satisfies $\nabla I=0$. For a fixed point $x\in X$ consider the monodromy group $\Mon(\nabla)\subset\GL(T_xX)$. Suppose that there exists a lattice $\Lambda_x\subset T_xX$ in a fiber which is preserved by the action of the monodromy group: $\Mon(\nabla)\Lambda_x=\Lambda_x$. Then we can construct the set $\Lambda\subset TX$ via all parallel transportation of $\Lambda_x$. Define a manifold $X^{+}:=TX/\Lambda$. It is fibered over $X$ with fibers $T_xX/\Lambda_x$ are compact tori. It makes sense since for each $x\in X$ the intersection $\Lambda\cap T_xX$ is a lattice continuously depending on $x\in X$. The manifold $X^{+}=TX/\Lambda$ is called {\bf the quaternionic double}. It was shown in \cite{SV} that there is a pair of  almost complex structures

$$I:=\begin{pmatrix}
I_X & 0 \\
0 & -I_X 
\end{pmatrix},\quad J:=\begin{pmatrix}
0 & -\Id \\
\Id & 0
\end{pmatrix}$$
on $X^{+}$ which are integrable and satisfy the quaternionic relations. 

\medskip
\begin{theorem}\label{SV}(Soldatenkov, Verbitsky)
Let $X^+$ be the quaternionic double of an affine complex manifold $X$. If $X$ is non-Ka\"hler, then $X^+$ does not admit an HKT-metric \cite{SV}.
\end{theorem}

It was shown in \cite{DF}, see also \cite{BDV}, that any abelian hypercomplex nilmanifold is HKT.

\medskip
\example
To provide an example of an $\H$-solvable Lie algebra with a non-abelian hypercomplex structure, we first define the Kodaira surface, following \cite{Has}, see also \cite[Example 1.7]{AV}.
Consider the Lie algebra $\g=\langle x, y, z, t \rangle$, such that the only non-zero commutator is $[x,y]=z$. The complex structure is given by $Ix=y$ and $Iz=t$. There exists an operator $\nabla^+:\g\times\g\arrow\g$ defined by the formula $$\nabla^+_ab:=\frac{1}{2}([a,b]+I[Ia,b]), a,b\in\g$$
We extended $\nabla^+$ to a left-invariant connection on the Lie group $G$, also denoted by $\nabla^+$. 
It is easy to see that the connection $\nabla^+$ is complex-linear, torsion-free and flat. Therefore, we could define the quaternionic double of the  Kodaira surface. It is a nilmanifold associated with the Lie algebra $\g^+=\g\oplus\g$, with the commutator defined as follows: $$[(a,b),(c,d)]:=([a,b],\nabla^+_ad-\nabla^+_cb).$$
The hypercomplex structure on $\g^+$ is defined as follows: $$I(a,b)=(Ia,-Ib),\quad J(a,b)=(-b,a),\quad K(a,b)=(-Ib,-Ia).$$
Then $$\g^+_1:=[\g^+,\g^+]=[\g\oplus\g,\g\oplus\g]=\langle\lambda z, \mu z\rangle, \lambda,\mu\in\R.$$
Hence $\g_2^+=0$ (because there are no non-trivial commutators on the second step), which implies the $\H$-solvability of $\g^+$. From \ref{SV} it is clear that the hypercomplex structure on $\g^+$ is non-abelian.

This gives an example of an $\H$-solvable Lie algebra with a non-abelian complex structure.


\medskip
\example Let $(N,I,J,K)$ be a hypercomplex manifold. Fix a point $p\in N$ and consider the set $\g^{(d)}\subset\g$ of smooth vector fields such that $X\in\g^{(d)}$ has zero of order $d$ at $p$. Notice that the filtration $\g^{(i)}$ is $\H$-invariant with the commutator $[\g^{(d)},\g^{(k)}]\subset\g^{(d+k-1)}$. Therefore, the quotient  $\g^{(2)}/\g^{(n)}$ is an $\H$-solvable algebra.

\medskip
{\bf Acknowledgments:} I am thankful to Misha Verbitsky for turning my attention to this problem, his support and attention during the preparation of this paper. 

\section{Preliminaries}
\subsection{Nilpotent Lie algebras}

Let $G$ be a real nilpotent Lie group, and $\g$ its  Lie algebra. {\bf The descending central series} of a Lie algebra $\g$ is the chain of ideals defined inductively:
\begin{equation*}
  \g_0\supset\g_1\supset\dots\supset\g_k\supset\dots,  
\end{equation*}
where $\g_0=\g$ and $\g_k=[\g_{k-1},\g]$. It is also called {\bf the lower central series} of $\g$.
 
\medskip
\begin{definition}
A Lie algebra $\g$ is called {\bf nilpotent} if $\g_k=0$ for some $k\in\mathbb{Z}_{>0}$.
\end{definition}

\medskip
Let $\g$ be a Lie algebra and $\g^*$ its dual space. Recall that for any $\alpha\in\g^*$ the Chevalley–Eilenberg  differential $d:\g^*\arrow\Lambda^2\g^*$ is defined as follows $$d\alpha(\xi,\theta)=-\alpha([\xi,\theta]),$$ 
where $\xi,\theta\in\g$.
It extends to a finite-dimensional complex
\begin{equation}\label{CH}
    0\arrow\g^*\arrow\Lambda^2\g^*\arrow\cdots\arrow\Lambda^{n}\g^*\arrow 0
\end{equation}
by the Leibniz rule: $d(\alpha\wedge\beta)=d\alpha\wedge\beta+(-1)^{\tilde{\alpha}}\alpha\wedge d\beta$, where $\alpha,\beta\in\g^*$ and $n=\dim_{\mathbb{R}}\g$. The condition $d^2=0$ is equivalent to the Jacobi identity.

\subsection{A short review of the Maltsev theory }

Following Maltsev's papers \cite{Mal} and \cite{Mal2} we are going to consider only nilpotent groups without torsion.  

\medskip
\begin{definition}
 {\bf A nilmanifold} is a  compact manifold which admits a transitive action of a nilpotent Lie group.
\end{definition}

\medskip
Recall that {\bf a lattice $\Gamma$} is a discrete subgroup of a topological group $G$ such that there exists a regular finite $G$-invariant measure on the quotient $\Gamma\backslash G$.

\medskip
The famous {\bf Maltsev's theorem } states that any nilmanifold $N$ is diffeomorphic to a quotient of a connected, simply connected nilpotent Lie group $G$ by a cocompact lattice $\Gamma$. The group $\Gamma$ is isomorphic to the fundamental group $\pi_1(N)$
of the nilmanifold $N$ and $G$ is the {\bf Maltsev completion} of the group $\Gamma\approx\pi_1(N)$ \cite{Mal}. 

\medskip
\begin{definition}
 A group $G$ is called {\bf complete} if for each $g\in G$ there exists $n\in\Z_{>0}$ such that the equation $x^n=g$ has solutions in $G$. 
\end{definition}

\medskip
\begin{definition}
 Let $G$ be a subgroup of a complete nilpotent group $\hat{G}$. Suppose that for any $g\in\hat{G}$ there exists $n\in\Z_{>0}$ such that $g^n\in G$. Then $\hat{G}$ is called {\bf the Maltsev completion} of a group $G$.
\end{definition}

\medskip
Let $\g_\mathbb{Q}$ be a nilpotent Lie algebra over the field of rational numbers $\mathbb{Q}$. We identify $\g_\mathbb{Q}$ with a subspace of a real nilpotent Lie algebra $\g=\g_{\Q}\otimes\R$, and call $\g_{\Q}$ {\bf the rational lattice }of $\g$.

\medskip
\begin{definition}
{\bf A rational structure} in a real nilpotent Lie algebra $\g$ is a rational lattice $\g_{\Q}$ such that $\g\cong\g_{\Q}\otimes\R$.
\end{definition}

\medskip
\begin{definition}
 Let $\Gamma$ be a lattice in a connected, simply connected nilpotent Lie group $G$. Then its {\bf associated rational structure} is the $\Q$-span of $\log\Gamma\subset\g$, where $\g=\Lie G$ is the Lie algebra. If $\g$ has a rational structure related to a $\Q$-algebra $\g_{\Q}\subset\g$ then there exists a discrete subgroup $\Gamma$ such that $\log\Gamma\subset\g_{\Q}$ and the quotient $\Gamma\backslash G$ is compact \cite[Theorem 5.1.7]{CG}.
\end{definition}

Let $\Gamma$ be a discrete subgroup of a connected, simply connected Lie group $G$. By \ref{Mal} the Maltsev completion $\hat{\Gamma}$ is the unique closed connected subgroup of $G$ such that a left quotient $\Gamma\backslash\hat\Gamma$ is compact.

\medskip
\begin{theorem}\label{Mal}(Maltsev)
Let $G$ be a finitely generated nilpotent group without torsion. Then there exists a nilpotent, complete and torsion-free group $\hat{G}$ such that $\hat{G}$ is a completion of $G$. Moreover, all the completions of $G$ are isomorphic \cite{Mal2}. Finally, $\hat{G}=\exp{\g_{\Q}}$ is the set of rational points in a real nilpotent Lie group $G$ with the Lie algebra $\g$ admitting a rational lattice $\g_{\Q}$.
\end{theorem}\endproof


\subsection{Hypercomplex nilmanifolds}

Let $X$ be a smooth manifold. Recall that {\bf an almost complex structure} on $X$ is an endomorphism $I\in\End(TX)$ satisfying $I^2=-\Id$. The Nijenhuis tensor $N_I$ associated to the almost complex structure $I$ is given by the formula $$N_I(X,Y)=[X,Y]+I[IX,Y]+I[X,IY]-[IX,IY].$$ An almost complex structure is called {\bf integrable} if its Nijenhuis tensor vanishes.

\remark $N_I=0$ if and only if $[TX^{1,0},TX^{1,0}]\subset TX^{1,0}$.

\medskip
\begin{theorem}(Newlander--Nirenberg)
If $I$ is an integrable almost complex structure on $X$, then $X$ admits the structure of a complex manifold compatible with $I$.
\end{theorem}

\medskip
\begin{definition}
 {\bf A complex nilmanifold} is a pair $(N,I)$, where $N=\Gamma\backslash G$ is a nilmanifold obtained from a nilpotent Lie group $G$ and $I$ an integrable left-invariant almost complex structure on $G$.
\end{definition}

By definition, $I$ is left-invariant if the left translations $L_g: (G,I)\arrow (G,I)$ are holomorphic. Notice that the Lie group $G$ does not need to be a complex Lie group, but in the case when it does both left and right translations on $G$ are holomorphic.

\medskip
Let $X$ be a smooth manifold equipped with three integrable almost complex structures  $I, J, K\in\End(TX)$, satisfying the quaternionic relations $I^2=J^2=K^2=-\Id$ and  $IJ=K=-JI$. Such a quadruple $(X,I,J,K)$ is called {\bf a hypercomplex manifold}. Obata \cite{Ob} proved  that there exists a unique torsion-free connection $\nabla^{Ob}$ preserving the complex structures: 
\begin{equation*}
  \nabla^{Ob}I=\nabla^{Ob}J=\nabla^{Ob}K=0.  
\end{equation*}
The connection $\nabla^{Ob}$ is called {\bf the Obata connection}.

\medskip
A hypercomplex structure induces a complex structure $L=aI+bJ+cK$ for each $(a,b,c)\in\mathbb{R}^3$ such that $a^2+b^2+c^2=1$ and the set of such structures is identified in a natural way with $S^2\approx\mathbb{C}{\rm P}^1$. 

\medskip
Consider the product $X\times \mathbb{C}\rm{P}^1$, where $X$ is a hypercomplex manifold.
{\bf The twistor space} $\Tw(X)$ of the hypercomplex manifold $X$ is a complex manifold where the complex structure is defined as follows. 
 For any point $(x,L)\in X\times \mathbb{C}\rm{P}^1$ the complex structure on $T_{(x,L)}\Tw(X)$ is $L$ on $T_xX$ and the standard complex structure $I_{\mathbb{C}P^1}$ on $T_L\mathbb{C}\rm{P}^1$. This almost complex structure on the twistor space of a hypercomplex manifold is always integrable \cite{K}, \cite[Theorem 14.68]{Besse}. The space $\Tw(X)$ is equipped with the canonical holomorphic projection $\pi:\Tw(X)\arrow\mathbb{C}\rm{P}^1$. The fiber $\pi^{-1}(L)$ at a point $L\in\mathbb{C}\rm{P}^1$  is biholomorphic to the complex manifold $(X,L)$.
 
 \medskip
 \begin{definition} Let $\Gamma$ be a cocompact lattice in a  nilpotent Lie group $G$ with a left-invariant hypercomplex structure. Then the manifold $N=\Gamma\backslash G$ is called {\bf a hypercomplex nilmanifold}.
 \end{definition}




\subsection{Positive bivectors on a Lie algebra}

Consider a nilpotent Lie group $G$ with a left-invariant complex structure $I\in\End(TG)$. Recall that {\bf a complex structure operator} on a Lie algebra $\g$ can be given by a decomposition of the complexification $\g_{\mathbb{C}}=\g\otimes\mathbb{C}$ satisfying $\g_{\mathbb{C}}=\g^{1,0}\oplus\overline{\g^{1,0}}$, where $\g^{1,0}=\{X\,|\, X\in\g_{\mathbb{C}},\,I(X)=\sqrt{-1}X\}$ and $[\g^{1,0},\g^{1,0}]\subset\g^{1,0}$ by \ref{complex structure}.

\medskip
Denote the $k$-th exterior power of $\g^{1,0}$ (resp. $\g^{0,1}$) by $\Lambda^{k,0}\g$ (resp. $\Lambda^{0,k}\g)$. Consider the graded algebra of $(p,q)$-multivectors $\Lambda^*\g\otimes\C=\bigoplus_{p,q}\Lambda^{p,q}\g$.

\medskip
\begin{definition}
The elements of the space  $\Lambda^{1,1}\g\subset\Lambda^2\g$ are called {\bf (1,1)-bivectors} or just {\bf bivectors}.
\end{definition}

\medskip
A non-zero real bivector $\xi\in \Lambda^{1,1}\g$ is called {\bf positive} if for any non-zero $\alpha\in \Lambda^{1,0}\g^*$ one has $\xi(\alpha,I{\alpha})\geq 0$.

\medskip
The Lie bracket gives a linear mapping $\delta_1:\Lambda^2\g\arrow\g$, defined as $\delta_1:x\wedge y\mapsto[x,y]$. Such a mapping extends by the formula \eqref{LR} below to the finite-dimensional complex of {\bf $k$-multivectors}, i.e. $\delta_{m}:\Lambda^{m+1}\g\arrow\Lambda^m\g$
\begin{equation*}
     0\arrow\Lambda^{2n}\g\xrightarrow{\delta_{2n-1}} \cdots\arrow\Lambda^{2}\g\xrightarrow{\delta_1}\g\xrightarrow{\delta_0} 0
\end{equation*}
and it is dual to the Chevalley-Eilenberg complex \eqref{CH}. The boundary operator $\delta_{k-1}$ can be written as follows
\begin{equation}\label{LR}
    \delta_{k-1}(x_1\wedge\dots\wedge x_k)=\sum_{r<s}(-1)^{r+s+1}[x_r,x_s]\wedge x_1\wedge\dots\wedge\hat{x}_r\wedge\cdots\wedge\hat{x}_s\wedge\cdots\wedge x_k.
\end{equation}

\medskip
\begin{definition} 
 {\bf A complex curve} in a complex manifold $(X, I)$ is a 1-dimensional compact complex subvariety $C_I\subset X$.
\end{definition}

\medskip
Let $C_I\subset N$ be a complex curve in a complex nilmanifold $(N,I)$ and $\omega\in\Lambda^2\g^*$ a two-form. We identify $\Lambda^2\g^*$ with the space of left-invariant 2-forms on the Lie group $G$, which descends to the space of 2-forms $\Lambda^2(N)$ on the nilmanifold $N=\Gamma\backslash G$.

Consider a functional $\xi$ on the space of 2-forms $\Lambda^2\g^*$:
\begin{equation}\label{current}
    \xi_{C_I}(\omega):=\int_{C_I}\omega.
\end{equation}
Such a functional defines a bivector $\xi\in\Lambda^2\g$\footnote{Since the homology $H_*(N)=H_*(\g)$ by \ref{Nomizu} a complex curve $C_L$ corresponds to the bivector $\xi_{C_L}$}.

\section{Positive bivectors on a quaternionic vector space}

We start with a sequence of linear-algebraic lemmas. Let $V$ be a finite-dimensional vector space over $\C$ and $V^*$ its dual.
Denote by $(V,I)$ the pair of a vector space $V$ with a complex structure $I\in\End(V)$ on it.

\medskip
Recall that {\bf the kernel} of a bivector $\xi\in\Lambda^{1,1}V$ is the following set:
\begin{equation}\label{kernel}
   \ker\xi=\{x\in V^*\,|\,\xi(x,\cdot\,)=0\}\subset V^*. 
\end{equation}

We denote the space of positive bivectors with respect to the complex structure $I$ on a vector space $V$ by $\Lambda^{1,1}_{I, pos}V$.

\medskip
\begin{lemma}\label{Ker}
Let $(V,I)$ be a vector space with a complex structure $I$ and $\xi\in\Lambda^{1,1}_{I,\rm{pos}}V$ a non-zero positive bivector. Let $V_1^*:=\{x\in V^*\,|\, \xi(x,Ix)=0\}$. Then $V_1^*=\ker\xi$.
\end{lemma}

\begin{proof} From the definition \eqref{kernel} it is obvious that $\ker\xi\subset V_1^*$. 
Suppose that $x\in V_1^*$ and $x\not\in\ker\xi$. Then $0\ne [x]\in V/\ker\xi$. On the space $V/\ker\xi$ the bivector $\xi$ is positive definite because it has no kernel and it is diagonalizable. 
Hence, $\xi(x,Ix)>0$, which is a contradiction.
\end{proof}

\medskip 
Let $W_1$ be a subspace of a vector space $(W,I)$ and consider two maps: 
$p:W\arrow W/W_1$ and  $\tilde{p}:\Lambda^2W\arrow\Lambda^2(W/W_1)$. There are two subspaces of $\Lambda^2W$: $\Lambda^2\ker p=W_1\wedge W_1$ and $\ker\tilde{p}=W\wedge W_1$. It is obvious that $\Lambda^2\ker p\subset\ker\tilde{p}$. We are going to show that $\Lambda^{1,1}_{I,pos}W_1\cap\ker\tilde{p}=\Lambda^{1,1}_{I,pos}W_1\cap\Lambda^2\ker p$.

\medskip
\begin{lemma}\label{p_1}
Let $W_1$ be a subspace of a vector space $(W,I)$ and $\xi\in\Lambda^{1,1}_{I, pos}W$ a positive (1,1)-bivector. Assume that $\xi\in\ker\tilde{p}$. Then $\xi\in\Lambda^2W_1$. 
\end{lemma}

\begin{proof}
Denote by $W_1^{\perp}\subset W^*$ the annihilator of the subspace $W_1$; it is isomorphic to the dual of the quotient $(W/W_1)^*$.
Since $\xi\in W\wedge W_1$, we have $\xi(W_1^{\perp},W_1^{\perp})=0$. Therefore,  $W_1^{\perp}\subset\ker\xi$ by \ref{Ker}. 
So, $\xi_{\arrowvert_{W^*\wedge W_1^{\perp}}}=0$, which implies that $\xi\in \Lambda^2W_1$.
 \end{proof}

\medskip
Recall that for any pair of orthogonal complex structures $I,J\in\H$, one has
\begin{equation}\label{ocs}
    J(\Lambda^{p,q}_IV)=\Lambda^{q,p}_IV,
\end{equation}
 where the action of the complex structures $I$ and $J$ extended from $\Lambda^{1,0}V$ and $\Lambda^{0,1}V$ to $(p,q)$-bivectors by a multiplicativity. Indeed, $I$ and $J$ anticommute on $\Lambda^1(V)$ which implies $J(\Lambda^{1,0}_I)=\Lambda^{0,1}_I$ and 
$J(\Lambda^{0,1}_I)=\Lambda^{1,0}_I$

Consider the operator $W_I:\Lambda^*V_I\arrow\Lambda^*V_I$ defined by the formula $W_I(\xi)=\sqrt{-1}(p-q)\xi$, where $\xi\in\Lambda^{p,q}_IV$. Notice that the elements $W_I, W_J$ and $W_K$ generate the Lie algebra $\su(2)$ and the complex structures $I, J$ and $K$ are the the elements of the Lie group $SU(2)$ related to them, $I=\exp{\frac{\pi W_I}{2}}, J=\exp{\frac{\pi W_J}{2}}$ and $K=\exp{\frac{\pi W_K}{2}}$ \cite{V2}.

\medskip
\begin{lemma}\label{IJ}
Let $V$ be a quaternionic vector space and $\Lambda^{1,1}_{I,\rm{pos}}V\subset\Lambda^2V$ the space of positive $(1,1)$-bivectors on $(V,I)$. Then $\Lambda^{1,1}_{I,\rm{pos}}V\cap\Lambda^{1,1}_{I',\rm{pos}}V=0$ for distinct complex structures $I$ and $I'$.
\end{lemma}

\begin{proof}
 Denote the intersection $R_{I,I'}:=\Lambda^{1,1}_{I,\rm{pos}}V\cap\Lambda^{1,1}_{I',\rm{pos}}V$.
First, suppose that $I'=-I$, the non-zero bivector $\xi\in R_{I,-I}$, and let $\alpha\in V^*$. Then $\xi(\alpha,I\alpha)> 0$ and $\xi(\alpha,-I\alpha)< 0$, so there is no such a bivector $\xi$.

Assume that $I'\ne -I$.
 Then suppose that $J\in\H$ is orthogonal to $I$ and $IJ=-JI$, $J^2=-\Id$. It is clear that as $I\ne\pm I'$, $I$ is not proportional to $I'$, then $I'$ can be written in a form $I'=aI+bJ$ for some $a,b\in\R$, $b\ne 0$.
 
 Let $\xi\in R_{I,I'}$. Since $\xi$ is a $(1,1)$-bivector, $W_I(\xi)=W_{I'}(\xi)=0$. Hence, $W_K(\xi)=0$ because $W_I$ and $W_{aI+bJ}$ generate the Lie algebra $\su(2)$ \cite{V2}. We obtain that $\xi$ is an $\su(2)$-invariant bivector. It is therefore invariant under the multiplicative action $J(\xi)(\alpha,\beta):=\xi(J\alpha,J\beta)$ of $J\in SU(2)$. Consider
\begin{equation}\label{SU}
  0\leq\xi(\alpha,I\alpha)=\xi(J\alpha,JI\alpha)=-\xi(J\alpha,IJ\alpha)=-\xi(\beta,I\beta),  
\end{equation}
where $\beta=J\alpha$. However, $-\xi(\beta,I\beta)\leq 0$, hence $\xi=0$. \end{proof}

\medskip
\corollary\label{PandI}
The intersection of the set of positive bivectors and $SU(2)$-invariant bivectors contains only zero bivector.

\proof Follows from the formula \eqref{SU}. An invariant bivector has to be positive for different complex structures, which is impossible by \ref{IJ}. \endproof


\section{Homology of a leaf of a foliation}
Recall that a CW-space $X$ with the fundamental group $\pi_1(X)=\pi$ and the higher homotopy groups $\pi_i(X)=0$ for $i>1$ is called  {\bf a $K(\pi, 1)$-space of Eilenberg–MacLane} or just {\bf $K(\pi, 1)$-space}. Since the universal covering of a nilpotent Lie group is contractible, the nilmanifold $N=\Gamma\backslash G$ with the fundamental group $\pi_1(N)\approx\Gamma$ is a $K(\Gamma, 1)$-space. {\bf The cohomology} of the group $\Gamma$ is defined as $H^*(K(\Gamma,1),\Q)$.

\medskip
In \cite{N} Nomizu showed that the de Rham cohomology of a nilmanifold $N=\Gamma\backslash G$ can be computed using the left-invariant differential forms on the Lie group $G$.

\medskip
\begin{theorem}({\bf Nomizu}, \cite[Theorem 1]{N})\label{Nomizu}
Let $N=\Gamma\backslash G$ be a nilmanifold and $(\Lambda^*\g^*,d)$ is the Chevalley–Eilenberg complex. The natural inclusion of the complex of the left-invariant differential forms $\Omega^{inv}(G)$ on the nilpotent Lie group $G$ into the de Rham algebra on the nilmanifold $\Omega^{inv}(N)$ induces the isomorphism of the corresponding cohomology  $H^*(\g,\R)\approx H^*(N,\R)$ . 
\end{theorem}\endproof

Since the nilmanifold $N$ is $K(\Gamma,1)$, the homology $H_*(N,\R)\approx H_*(\Gamma,\R)$, hence $H_*(\Gamma,\R)\approx H_*(\g,\R)$. Pickel showed that instead of real coefficients we can take the rational ones, i.e. 
$H^*(\Gamma,\Q)\approx H^*(\g,\Q)$ as well \cite{P}.

\medskip
Consider a nilpotent Lie group $G$ and let $\Gamma$ be its discrete subgroup, $\hat{\Gamma}\subset G$ the Maltsev completion of $\Gamma$, and define the Lie group $\hat\Gamma_{\R}:=\exp({\log(\hat\Gamma)\otimes\R})$. By \ref{Mal}, $\Gamma$ is a lattice in $\hat{\Gamma}_{\R}$.
Since the quotients $\Gamma\backslash G$ and $\Gamma\backslash\hat\Gamma_{\R}$ are both $K(\Gamma,1)$, we have 
    $H_*(\Gamma\backslash G)=H_*(\Gamma\backslash\hat\Gamma_{\R})$.
From \ref{Nomizu} follows that $H_*(\Gamma\backslash\hat\Gamma_{\R})=H_*(\Lie\hat\Gamma_{\R})$, where $\Lie\hat\Gamma_{\R}\subset\g$ is the Lie algebra of $\hat\Gamma_{\R}$ and $H_*(\Lie\hat\Gamma_{\R})$ denotes the cohomology of the Chevalley–Eilenberg homology complex \eqref{CH}.

\medskip
Let $N=\Gamma\backslash G$ be a nilmanifold, $\g$ the Lie algebra of the Lie group $G$
and $\goth{f}\subset\g$ a Lie subalgebra. Let ${F}:=\exp{\goth{f}}\subset G$ be the corresponding Lie group.   For each $x\in{G}$ define a subgroup of the lattice $\Gamma$ as follows:
\begin{equation}\label{leaf}
    \Gamma_x=\{\gamma\in\Gamma\,|\,x\gamma x^{-1}\in{F}\}.
\end{equation}
In other words, $\Gamma_x=\Gamma\cap x^{-1}Fx$.

\medskip
Recall that {\bf a distribution} on a smooth manifold $X$ is a sub-bundle $\Sigma\subset TX$. The distribution called {\bf involutive} if it is closed under the Lie bracket. {\bf A leaf} of the involutive distribution $\Sigma$ is the maximal connected, immersed submanifold $L\subset X$ such that $TL=\Sigma$ at each point of $L$.  The set of all leaves is called {\bf a foliation}.

The algebra $\goth{f}$ defines a left-invariant foliation $\Sigma$ on $G$. The leaves $\goth{L}_x$ of the corresponding foliation on $\Gamma\backslash G$ are diffeomorphic to $\Gamma_x\backslash{xF}$ for each $x\in G$.

\medskip
\definition
A subalgebra $\f\subset\g$ is said to be {\bf rational} with respect to a given rational structure $\g_{\Q}$ on $\g$ if $\f_{\Q}:=\g_{\Q}\cap\f$ is a rational structure for $\f$, i.e. $\f=\f_{\Q}\otimes\R$.

\medskip
\remark
 {\bf The rational homology of the leaf $\goth{L}_{x}$} of the foliation $\Sigma$ is equal to
 \begin{equation}\label{homology}
     H_*(\goth{L}_x,\Q):=H_*(\Gamma_x\backslash xF,\Q)=H_*(\f_{\Q}),
    \end{equation} 
 where $H_*(\f_{\Q})$ is the homology of the complex dual to the rational Chevalley–Eilenberg complex \eqref{CH}. The last equality makes sense because of \ref{Nomizu} and \cite{P}.

 Consider the natural map of homology
\begin{equation}\label{mapj}
    j:H_*(\goth{L}_x,\Q)\arrow H_*(N,\Q),
\end{equation}
 associated with the immersion $\goth{L}_x\arrow N$. Notice that $j$ does not have to be injective.

\medskip
\begin{claim}\label{claim} Let $N$ be a nilmanifold and
$X\subset N$ a subvariety tangent to the foliation $\Sigma$ generated by the left translates of a Lie subalgebra $\f\subset\g$.
Then the fundamental class $[X]\in H_*(N,\Q)$ belongs to the image of $H_*(\f_{\Q})$ in $H_*(N,\Q)$, where the map $\tau:H_*(\f_{\Q})\arrow H_*(N,\Q)$ is obtained from \eqref{homology} and \eqref{mapj}.
\end{claim}

\proof
Let $X$ be a subvariety in a leaf of the foliation $\Sigma\subset TN$ and let ${F}:=\exp{\f}$. A leaf of $\Sigma$ is diffeomorphic to $\Gamma_x\backslash{xF}$, hence $[X]\in \tau(H_*(\Gamma_x\backslash xF,\Q))$, where we identified $H_*(\Gamma_x\backslash xF,\Q)=H_*(\f_{\Q})$. 
\endproof


\section{Finale}
Let $(\g,I,J,K)$ be a hypercomplex nilpotent Lie algebra. 
Define inductively $$\g_i^{\H}:=\H[\g_{i-1}^{\H},\g_{i-1}^{\H}],$$ where $\g^{\H}_1=\H[\g,\g]$ and let  $\ag_{i}:={\g^{\H}_{i-1}}/{\H[\g^{\H}_{i-1},\g^{\H}_{i-1}]}$ be the corresponding commutative quotient algebra, $i\in\Z_{>0}$.

\medskip
Observe that for any commutative Lie algebra $\ag$ its second homology group coincides with the space of all bivectors, $H_2(\ag,\R)=\Lambda^2\ag$. Denote by $\Lambda^{1,1}_{L,pos}\ag$ the set of positive $(1,1)$-bivectors with respect to the complex structure $L$.

\medskip
\begin{proposition}\label{Pos}
Let $\ag$ be a commutative hypercomplex Lie algebra, and $\s\subset\Lambda^2\ag$ a countable set of non-zero bivectors.
Then for all $L\in\C\rm{P}^1$ except at
most a countable number, the intersection
$\Lambda^{1,1}_{L,pos}\ag\cap\s=\emptyset$.
\end{proposition}

\proof
By \ref{IJ} for any non-zero $\xi\in\s$ there exists at most one complex structure $L_{\xi}\in\C{\rm P}^1$ such that $\xi\in\Lambda^{1,1}_{L_{\xi},pos}\ag$. The union $\bigcup_{\xi\in\s}L_{\xi}$ is at most countable, hence for any $L\in\C{\rm P}^1\backslash\bigcup_{\xi\in\s}L_{\xi}$ the intersection $\Lambda^{1,1}_{L,pos}\ag\cap\s$ is empty.
\endproof


\medskip
Let $\Sigma_{i}$ be the foliation on a nilmanifold $N$ generated by the left-translates of the Lie subalgebra $\g_i^{\H}\subset\g={T}_eG$ and $\goth{L}_{x,i}\subset N$ a leaf of the foliation $\Sigma_i$. The leaf $\goth{L}_{x,i}$ is diffeomorphic to the left quotient $\Gamma_x\backslash xF_i$, where $F_i=\exp{\g_i^{\H}}\subset G$.

\medskip
Consider the natural projection $$p_i:\g_{i-1}^{\H}\arrow\ag_{i}.$$ Let $r_i$ be the corresponding map of the second homology:
\begin{equation*}
    r_i:H_2(\g_{i-1}^{\H})\arrow {H}_2(\ag_{i})=\Lambda^2\ag_i.
\end{equation*}
Then the image of a homology class in $H_2(\g^{\H}_{i-1})$  defines a bivector on the commutative Lie algebra $\ag_i$.

Denote by $\g^{\H}_{i,\Q}=\g^{\H}_{i}\cap\g_{\Q}$. Let $\s_i:=r_i(H_2(\g_{i-1,\Q}^{\H}))\subset\Lambda^2\ag_i$ and $R\subset\C{\rm P}^1$ be a union of
\begin{equation}\label{defR_i}
    R_i:=R[\s_i]\subset\C\rm{P}^1
\end{equation}
 the set of complex structures $L$ such that there exists a positive bivector $\xi\in\s_i\cap\Lambda^{1,1}_{L,pos}\ag$. 
By \ref{Pos}, the set $R_i$ is countable.

\medskip
\definition
Let $\Sigma_k$ be a holomorphic foliation obtained from the Lie subalgebra $\g^{I}_k=\g_k+I\g_k$. {\bf A transversal K\"ahler form}  $\omega_{k}$ with respect to the holomorphic foliation $\Sigma_k$ is a closed positive (1,1)-form, such that $\ker\omega_{k}$ is precisely the tangent space of the foliation, i.e. $\omega_{k}(\Sigma_k)=0$.

\medskip
\proposition\label{fin1}
 Let $C_L$ be a complex curve in a complex nilmanifold $(N,L)$, where $L\in\C{\rm P}^1\backslash R_i$ and the set $R_i\subset\C{\rm P}^1$ is defined in \eqref{defR_i}. Suppose that $C_L$ is tangent to the foliation $\Sigma_{i-1}$ defined by $\g^{\H}_{i-1}$ as above. Then it is also tangent to $\Sigma_{i}$.

\proof 
From \ref{claim} it follows that the fundamental class $[C_L]\in H_2(N,\Q)$ of the curve $C_L$ belongs to $j(H_2(\goth{L}_{x,i-1},\Q))\subset H_2(N,\Q)$, where $j$ is the standard map on the rational second homology \eqref{mapj}. \ref{Nomizu} allows us to identify the
 fundamental class $[C_L]\in H_2(N,\Q)$ with the bivector \eqref{current} $\xi_{C_L}=:\xi$.
 Under the projection $r_{i}$ the
 fundamental class $[C_L]$ is mapped to the bivector $r_{i}(\xi)\in\Lambda^2\ag_{i}$. 
  From the definition of the set $R_i$ we know that $\xi\in\ker r_{i}$ and from \ref{p_1} follows that $\xi\in\Lambda^{1,1}_{L,pos}\ker p_{i}=\Lambda^{1,1}_{L,pos}\g_{i}^{\H}$.
  
  Suppose that $\omega_{i-1}\in r^*_i(\Lambda^{1,1}_{L}\ag_{i}^*)$ is a transversal K\"ahler form of the foliation $\Sigma_{i-1}$. Then $\int_{C_L} \omega_{i-1}>0$ unless $C_L$ lies in the leaf of the foliation $\Sigma_{i-1}$.
  However, $\int_{C_L} \omega=0$ because $\omega_{i-1}$ is closed (otherwise, referring to the analogue of Stokes' theorem, we obtain that the volume of a compact manifold is equal zero).
  Since $\g^{\H}_{i-1}\supset\g^{\H}_i$ we have $\omega_{i-1}\in\Lambda^{1,1}_{L,pos}(\g^{\H}_{i-1})^*\subset\Lambda^{1,1}_{L,pos}(\g^{\H}_{i})^*$.
  Hence, $\g^{\H}_i\subset\ker\omega_{i-1}$. Hence, $\omega_{i-1}$ is a transversal K\"ahler form with respect to the foliation $\Sigma_i$ and $C_L$ lies in a leaf of $\Sigma_i$. \endproof

 
\medskip
Assume that for some $k\in\Z_{>0}$ the following sequence terminates: 
\begin{equation}\label{gH}
 \g_1^{\H}\supset\g_2^{\H}\supset\cdots\supset\g_{k-1}^{\H}\supset\g_k^{\H}=0,
\end{equation}
i.e. the Lie algebra $\g$ is $\H$-solvable, see also \ref{H-sol}.

\medskip
\corollary\label{cor} Let $L\in\C{\rm P}^1\backslash R$, where $R=\bigcup R_i$ is the countable subset defined in \eqref{defR_i}, and assume that the sequence \eqref{gH} terminates to zero. Then the complex nilmanifold $(N,L)$ contains no complex curves.

\proof Suppose that the sequence \eqref{gH} vanishes on the $k$-th step, i.e. $\Sigma_k=\{0\}$. Then \ref{cor} follows from the \ref{fin1} and the induction on $i$. \endproof

\medskip
\theorem\label{fin2}
Let $(N, I, J, K)$ be a hypercomplex nilmanifold and assume that the corresponding Lie algebra is $\H$-solvable. Then there are no complex curves in the general fiber of the holomorphic twistor projection $\Tw(N)\arrow\C\rm{P}^1$.

\proof
Follows from \ref{cor}. \endproof


\noindent {\sc Yulia Gorginyan\\
{\sc Instituto Nacional de Matem\'atica Pura e
              Aplicada (IMPA) \\ Estrada Dona Castorina, 110\\
Jardim Bot\^anico, CEP 22460-320\\
Rio de Janeiro, RJ - Brasil\\
also:\\
{\sc Laboratory of Algebraic Geometry,\\
National Research University (HSE),\\
Department of Mathematics, 6 Usacheva Str.\\ Moscow, Russia}\\
\tt  ygorginyan@hse.ru }\\
}

\end{document}